\documentclass[a4paper,12pt,oneside]{article}

\usepackage[latin1]{inputenc}
\usepackage[english]{babel}
\usepackage{amsthm}
\usepackage{amssymb}
\usepackage{amsmath}
\usepackage{hyperref}
\usepackage{mathrsfs}

\usepackage{pstricks}
\usepackage[nottoc]{tocbibind}

 \usepackage{tikz}

\usepackage{graphicx}
\graphicspath{{images/}}

\usepackage{longtable,geometry}
\usepackage{color}

\usepackage{hyperref}

\newtheoremstyle{note}{12pt}{12pt}{}{}{\bfseries}{.}{.5em}{}
\title{\LARGE\textbf{Unbounded Regime for Circle Maps with a Flat Interval}}
\author{Liviana Palmisano\\ Institute of Mathematics of PAN\\
Warsaw, Poland}

\newtheorem{theorem}[equation]{Theorem}
\newtheorem{proposition}{Proposition}

\newtheorem{definition}[equation]{Definition}

\newtheorem{remark}[equation]{Remark}
\newtheorem{corollary}[equation]{Corollary}
\newtheorem{fact}[equation]{Fact}
\numberwithin{equation}{section}
\newtheorem{lemma}[equation]{Lemma}

\newcommand{\N}{{\mathbb N}}

\newcommand{\R}{{\mathbb R}}

\newcommand{\T}{{\mathbb T}^2}
\renewcommand{\S}{{\mathbb S}^1}

\newcommand{\Cinf}{{{\mathcal C}^\infty}}
\newcommand{\Cd}{{{\mathcal C}^2}}

\newcommand{\C}{{\mathcal C}}

\renewcommand{\L}{{\mathscr L}}

\newcommand{\Cr}{\operatorname{\mathbf{Cr}}}
\newcommand{\Poin}{\operatorname{\mathbf{Poin}}}

\newcommand{\dist}{\operatorname{dist}}
\newcommand{\Sing}{\operatorname{Sing}}

\usepackage{makeidx}

\usepackage{hyperref}

\begin{document}
\maketitle
\author

\begin{abstract}
We study $\Cd$ weakly order preserving circle maps with a flat interval. In particular we are interested in the geometry of the mapping near to the singularities at the boundary of the flat interval. Without any assumption on the rotation number we show that the geometry is degenerate when the degree of the singularities is less than or equal to two and becomes bounded when the degree goes to three. As an example of application, the result is applied to study Cherry flows.
\end{abstract}

\section{Introduction}
\subsection{Motivation}
The principal purpose of this paper is to study the dynamics of a class $\L$ of circle maps of degree one, supposed to be $\C^2$ everywhere with the exception of two points where they are continuous and such that they are constant on one of the two intervals delimited by these two points. Moreover on a half open neighborhood of these two points the maps can be written as $x^\ell$ where the real positive number $\ell$ is called the critical exponent of the function.
\par
The study of this kind of map has a long history (see \cite{veer}, \cite{MvSdMM}, \cite{VT1}, \cite{VT2}, \cite{mendes}, \cite{Morgan}, \cite{5aut}, \cite{my}). One of the reasons for their investigation is connected to the understanding of particular flows on the two-dimensional torus, called Cherry flows. In fact the first return map for a Cherry flow is a function belonging to the class $\L$ (for more details see \cite{MvSdMM}, \cite{ABZ}, \cite{my}, \cite{my4}). The first example of such a flow was given by Cherry in 1937 and still a lot of questions about metric, ergodic and topological properties of Cherry flows remain open.
\par
Moreover, this kind of functions also arise naturally in the theory of circle mappings themselves as upper or lower maps of degree $1$ transformations which are not homeomorphisms.
A rapid development of the theory of mappings with a flat interval occurred in the decade between $1985$ and $1995$  with the introduction of analytic tools based on cross-ratio distortion. Afterwards, a lull occurred due to a lack of new motivating questions. {This paper is a part of the recent reawakening of interest in this class of mappings, principally motivated by a deeper understanding of their connections with Cherry flows. This line of research has also been pursued recently for example in \cite{my} and \cite{SV}.}

\par

 In this paper we are interested in the study of the geometry of functions in $\L$  near to the boundary points of the flat interval. Without any assumption on the rotation number we discover a change of the geometry depending on the degree of the singularities at the boundary points of the flat interval. Expressly, we show that the geometry is bounded when the critical exponents of our maps become greater than $3$.
\par
This result makes a contribution to the theory of circle maps with a flat interval and  it is particularly interesting because it opens the way towards understanding metric and ergodic properties of Cherry flows.
\par
Before we can explain our results more precisely, it is necessary to define our class and fix some notation.
\subsection{Assumptions and Notations}\label{class}
\paragraph{Hypotheses.}
\begin{enumerate}
\item We consider continuous circle endomorphisms $f$ of degree one, at least twice continuously differentiable except for two points (endpoints of the flat interval).
\item The first derivative of $f$ is everywhere positive except for the closure of an open non-degenerate interval $U$ (the flat interval) on which it is equal to zero.
\item Let $\left(a,b\right)$ be a preimage of $U$ under the natural projection of the real line on $\S$.
On some right-sided neighborhood of $b$, $f$ can be represented as
\begin{displaymath}
h_{r}\left(\left(x-b\right)^{\ell}\right),
\end{displaymath}
where $h_{r}$ is a $C^{2}$-diffeomorphism on an open neighborhood of $b$.
Analogously, there exists a $C^{2}$-diffeomorphism on a left-sided neighborhood of $a$ such that $f$ is of the form
\begin{displaymath}
h_{l}\left(\left(a-x\right)^{\ell}\right).
\end{displaymath}
The real positive number $\ell$ is called the critical exponent of $f$.
\end{enumerate}
\par
In the future we will assume that $h_r(x)=h_l(x)=x$. It is in fact possible to make $\Cd$ coordinate changes near $a$ and $b$ that will allow us to replace both $h_r$ and $h_l$ with the identity function.
\par
The class of such maps will be denoted by $\L$.

\paragraph{Basic Notations.}
We will introduce a simplified notation for backward and forward images of the flat interval $U$. Instead of $f^{i}(U)$ we will simply write $\underline{i}$; for example, $\underline{0}=U$. Thus, for us, underlined positive integers represent points, and underlined non-positive integers represent intervals.
\paragraph {Distance between Points.}
We denote by $(a,b)=(b,a)$ the  shortest  open interval between $a$ and $b$ regardless of the order of these two points. The length of that interval in the natural metric on the circle will be denoted by $|a-b|$. Following \cite{5aut}, let us adopt these notational conventions:
\begin{itemize}
\item $|\underline{-i}|$ stands for the length of the interval $\underline{-i}$.
\item Consider a point $x$ and an interval $\underline{-i}$ not containing it. Then the distance from $x$ to the closer endpoint of $\underline{-i}$ will be denoted by $|(x,\underline{-i})|$, and the distance to the more distant endpoint by $|(x,\underline{-i}]|$.
\item We define the distance between the endpoints of two intervals $\underline{-i}$ and $\underline{-j}$ analogously. For example, $|(\underline{-i},\underline{-j})|$ denotes the distance between the closest endpoints of these two intervals while $|[\underline{-i},\underline{-j})|$ stands for $|\underline{-i}|+|(\underline{-i},\underline{-j})|$.
\end{itemize}
\paragraph{Rotation Number.}
As the maps we consider are continuous and weakly order preserving, they have a rotation number; this number is the quantity which measures the rate at which an orbit winds around the circle. More precisely, if $f$ is a map in $\L$ and $F$ is a lifting of $f$ to the real line, the rotation number of $f$ is the limit
\begin{displaymath}
\rho(f)=\lim_{n\to\infty}\frac{F^n(x)}{n} \textrm{ (mod $1$)}.
\end{displaymath}
This limit exists for every $x$ and its value is independent of $x$. Because the dynamics is more interesting, in the discussion that follows and for the rest of this paper we will assume that the rotation number is irrational. Also, it will often be convenient to identify $f$ with a lift $F$ and subsets of $\mathbb S^1$ with the corresponding subsets of $\mathbb R$.

\paragraph{Combinatorics.}

Let $f\in\L$ and let $\rho(f)$ be the rotation number of $f$. Then, $\rho(f)$ can be written as an infinite continued fraction
\begin{displaymath}
\rho(f)=\frac{1}{a_1+\frac{1}{a_2+\frac{1}{\cdots}}},
\end{displaymath}
where $a_i$ are positive integers.

If we cut off the portion of the continued fraction beyond the $n$-th position, and write the resulting fraction in lowest terms as $\frac{p_n}{q_n}$ then the numbers $q_n$ for $n\geq 1$ satisfy the recurrence relation
\begin{equation}\label{q_n}
q_{n+1} = a_{n+1}q_n+q_{n-1};\textrm{  } q_0 = 1;\textrm{  } q_1 = a_1.
\end{equation}

The number $q_n$ is the number of times we have to iterate the rotation by $\rho(f)$ in order that the orbit of any point makes its closest return so far to the point itself (see Chapter I, Sect. I in \cite{deMvS}).

\subsection{Discussion and Statement of the Results}
As stressed before, in this paper we are interested in the study of the geometry of functions in $\L$  near to the boundary points of the flat interval. This quantity is measured by the sequence of scalings \begin{displaymath}
\tau_{n}:=\frac{|(\underline 0,\underline{q_n})|}{|(\underline 0,\underline{q_ {n-2}})|}.
\end{displaymath}
When $\tau_n\rightarrow 0$ we say that the geometry of the mapping is `degenerate'.
When $\tau_n$ is bounded away from zero we say that the geometry is `bounded'.
\par
The same problem was analyzed in \cite{5aut} for functions in $\L$ with rotation number of bounded type  ($\sup_{i}a_i<\infty$) and with negative Schwarzian derivative\footnote{The Schwarzian derivative of a function $f$ is defined to be $Sf(x):=\frac{f^{'''}(x)}{f^{'}(x)}-\frac{3}{2}\left(\frac{f^{''}(x)}{f^{'}(x)}\right)^2$.}. This last assumption was then removed in \cite{my}.  In these papers, it is proved that the geometry is degenerate when the critical exponent is less than or equal to $2$ and becomes bounded when the critical exponent passes $2$. So, a phase transition occurs in the dynamics of the system depending on the degree of the singularities at the boundary points of the flat interval. 
\par
This result suggests to us the natural problem of investigating the unbounded regime. In this case it becomes more delicate to make conjectures; surprises often occur due to the presence of underlying parabolic phenomena. The main result we have obtained is the following:

\begin{theorem}\label{sw1}
Let $f$ be a function of the class $\mathscr{L}$ with critical exponent $\ell>1$:
\begin{enumerate}
\item If $\ell\leq 2$, then the sequence $(\tau_n)_{n \in \N}$ tends to zero at least exponentially fast.
\item If $\ell\geq3$, the sequence $(\tau_n)_{n \in \N}$ is bounded away from zero.
\end{enumerate}
\end{theorem}

Without any assumption on the rotation number we prove that the geometry near to the boundary points of the flat interval is degenerate when the critical exponent is less than or equal to $2$ and becomes  bounded when the critical exponent goes to $3$.  It remains unknown what happens between $2$ and $3$.
\par
The difficulty of the problem comes from the presence of parabolic phenomena that generate accumulation of constants which is not always easy to control. In fact the main idea of the proof is to find a recursive formula for the sequence $\tau_n$ and to study its convergence. The accumulation of constants appear basically everywhere: both in the recurrence as well as in the study of the convergence. This fact leads us to suspect that the problem could be real, not only technical.
\par
Moreover this result remains the first one to be valid for functions with any rotation number. It opens further questions and potentially has many interesting and significant applications.

{We illustrate it on an example of studies of the quasi-minimal set\footnote{The closure of any recurrent trajectories is called a quasi-minimal set} of a Cherry flow. We recall that Cherry flows are $\Cinf$ flows on the $2$-dimension torus with one hyperbolic saddle and sink. They were construct in 1938 by Cherry in \cite{Cherry} and they were the first example of $\Cinf$ flows on the torus with a non-trivial quasi-minimal set.}

\par
{Using Theorem \ref{sw1} and following the strategies in \cite{my}, we are now able to generalize Theorem 1.6 in \cite{my} and give an example of Cherry flow with a metrically non-trivial quasi-minimal set in the general case of unbounded regime \footnote{Any Cherry flow has a well defined rotation number $\rho \in [0,1)$  equal to the rotation number of its first return map to any global Poincar\'e section.}. More precisely:}

\begin{theorem}\label{3}
Let $X$ be a Cherry flow with  $\lambda_1>0>\lambda_2$ being the eigenvalues of the saddle point.  If  $\left|\lambda_2\right|\geq3\lambda_1$ then the quasi-minimal set of $X$ has Hausdorff dimension strictly greater than $1$.
\end{theorem}

\par
{Another application of Theorem \ref{sw1} concerns the study of the physical measures for Cherry flows. These are probability measures with basin of attraction of positive Lebesgue measure and they are of a particular interest as they describe the statistical properties of a large set of orbits. Such a study were initiated in \cite{SV}. While the non-positive divergence case was resolved, the positive divergence one still lacked the complete description. Some conjectures were put forward. Theorem \ref{sw1} gives an answer to this conjectures by providing a description of the physical measures for Cherry flows in the positive divergence case. The details are contained in \cite{my3}.}

\section{Technical Tools}

\subsection{Distortion
Techniques}\label{distortion}
The main ingredient in the proof of the principal result of this paper is the control of the distortion of iterates of maps in $\L$. We will use two different cross-ratios, $\Cr$ and $\Poin$.
\begin{definition}
If $a<b<c<d$ are four points on the circle, then we can define their \it{cross-ratio} $\Cr$ by:
\begin{equation*}
\Cr(a,b,c,d):=\frac{|b-a||d-c|}{|c-a||d-b|},
\end{equation*}
and their \it{cross-ratio} $\Poin$ by:
\begin{equation*}
\Poin(a,b,c,d):=\frac{|d-a||b-c|}{|c-a||d-b|}.
\end{equation*}
\end{definition}
\par
Now we analyze the distortion of these two kinds of cross-ratios.
\par

Diffeomorphisms with negative Schwarzian derivative increase cross-ratio \textbf{Poin}:
\begin{displaymath}
\textbf{Poin}\left(f(a),f(b),f(c),f(d)\right)>\textbf{Poin}\left(a,b,c,d\right).
\end{displaymath}
In general, without the assumption of negative Schwarzian, the following holds:
\begin{theorem}\label{point}
Let $f$ be a $C^2$ map with no flat critical points. There exists a bounded increasing function $\sigma:\left[0,\infty\right)\to\mathbb R_+$ with $\sigma(t)\to 0$ as $t\to 0$ with the following property. Let $\left[b,c\right]\subset\left[a,d\right]$ be intervals such that $f^n_{|\left[a,d\right]}$ is a diffeomorphism. Then
\begin{displaymath}
\textbf{Poin}\left(f^n(a),f^n(b),f^n(c),f^n(d)\right)\geq exp\{-\sigma(\tau)\sum_{i=0}^{n-1}|f^{i}(\left[a,b\right))|\}\textbf{Poin}\left(a, b, c, d\right),
\end{displaymath}
where $\tau=\max_{i=0,\dots,n-1}|f^{i}(\left(c,d\right])|$.
\end{theorem}
The proof of Theorem \ref{point} can be found in \cite{SvS}.

\par
Here, we formulate the result which enables us to control the growth of the iterates of cross-ratios $\Cr$ even if the map is no longer a homeomorphism with negative Schwarzian or is not invertible.\\
The reader can refer to \cite{Swiatek} for the general case and to \cite{Grfun} for our situation.
\par
Consider a chain of quadruples
\begin{equation*}
\bigcup_{i = 0}^n \lbrace{(a_i,b_i,c_i,d_i)\rbrace} 
\end{equation*}
such that each is mapped onto the next by the map $f$. If the following conditions hold:
\begin{itemize}
\item There exists un integer $k \in \N$, such that each point of the circle belongs to at most $k$ of the intervals $(a_i,d_i)$.
\item The intervals $(b_i,c_i)$ do not intersect $\underline 0$.
\end{itemize}
Then, there exists a constant $K > 0$, independent of the set of quadruples, such that:
\begin{equation*}
\log\frac{\Cr(a_n, b_n, c_n, d_n)}{\Cr(a_0, b_0, c_0, d_0)}\leq K
\end{equation*}

In order to control the distortion of the iterates of our maps we will also frequently use the following proposition which is a corollary of the Koebe principle in \cite{GS-Schwarz}.
\begin{proposition}\label{Koebe liv}
Let $f$ be a function in $\mathscr{L}$ and let $J\subset T$ be two intervals of the circle. Suppose that, for some $n\in\mathbb N$
\begin{itemize}
\item[-] $f^n$ is a diffeomorphism on $T$,
\item[-] $\sum_{i=0}^{n-1}\left|f^{i}(J)\right|$ is bounded,
\item[-] $\left|f^n\left(J\right)\right|\leq K dist\left(f^n\left(J\right),\partial f^n\left(T\right)\right)$ with $K$ a positive constant.
\end{itemize}
 Then, there exists a constant $C$ such that, for every two intervals $A$ and $B$ in $J$
\begin{displaymath}
\frac{\left|f^n\left(A\right)\right|}{\left|f^n\left(B\right)\right|}\geq C\frac{\left|A\right|}{\left|B\right|}.
\end{displaymath}

\end{proposition}

\subsection {Continued Fractions and Partitions}\label{dynamicalpartition}
Let $f\in\L$. Since $f$ is order-preserving and has no periodic points, there exists an order-preserving and continuous map $h:\mathbb{S}^{1}\rightarrow\mathbb{S}^{1}$ such that $h\circ f=R_{\rho}\circ h$, where $\rho$ is the rotation number of $f$ and $R_{\rho}$ is the rotation by $\rho$. In particular, the order of points in an orbit of $f$ is the same as the order of points in an orbit of $R_{\rho}$. Therefore, results about $R_{\rho}$ can be translated into results about $f$, via the semiconjugacy $h$.

We can build the so called dynamical partitions $\mathscr{P}_{n}$ of $\mathbb S^1$ to study the geometric properties of $f$, see \cite{Grfun}. The partition $\mathscr{P}_{n}$ is generated by the first $q_{n}+q_{n+1}$ preimages of $U$ and consists of
\begin{displaymath}
\mathscr{I}_{n}:=\left\{ \underline{-i}: 0\leq i\leq q_{n+1}+q_{n}-1\right\},
\end{displaymath}
together with the gaps between these intervals.

There are two kinds of gaps:
\begin{itemize}
\item The `long' gaps are of the form
\begin{displaymath}
I_{i}^{n}:=f^{-i}(I^{n}_{0}), i=0,1,\ldots, q_{n+1}-1
\end{displaymath}
where $I^{n}_{0}$ is the interval between $\underline{-q_{n}}$ and $\underline 0$ for $n$ even or the interval between $\underline{0}$ and $\underline{-q_n}$ for $n$ odd.
\end{itemize}
\begin{itemize}
\item The `short' gaps are of the form
\begin{displaymath}
I_{i}^{n+1}:=f^{-i}(I^{n+1}_{0}), i=0,1,\ldots, q_{n}-1
\end{displaymath}
where $I^{n+1}_{0}$ is the interval between $\underline 0$ and $\underline{-q_{n+1}}$ for $n$ even or the interval between $\underline{ -q_{n+1}}$ and $\underline{0}$ for $n$ odd.
\end{itemize}
We will briefly explain the structure of the partitions. Take two consecutive dynamical partitions of order $n$ and $n+1$. The latter is clearly a refinement of the former. All `short' gaps of $\mathscr P_{n}$ become `long' gaps of $\mathscr P_{n+1}$ while all `long' gaps of $\mathscr P_{n}$ split into $a_{n+2}$ preimages of $U$ and $a_{n+2}$ `long' gaps and one `short' gap of the next partition $\mathscr P_{n+1}$:
\begin{equation}\label{div}
I_{i}^{n}=\bigcup_{j=1}^{a_{n+2}}f^{-i-q_n-jq_{n+1}}(U)\cup\bigcup_{j=0}^{a_{n+2}-1}I_{i+q_n+jq_{n+1}}^{n+1}\cup I_{i}^{n+2}.
\end{equation}

Several of the proofs in the following will depend strongly on the relative positions of the points and intervals of $\mathscr{P}_{n}$. In reading the proofs the reader is advised to keep the Figure \ref{fig:primera} in mind, which show some of these objects near the flat interval $\underline{0}$.
 \begin{figure}[h]
		\begin{center}
		\includegraphics[width=15.5cm]{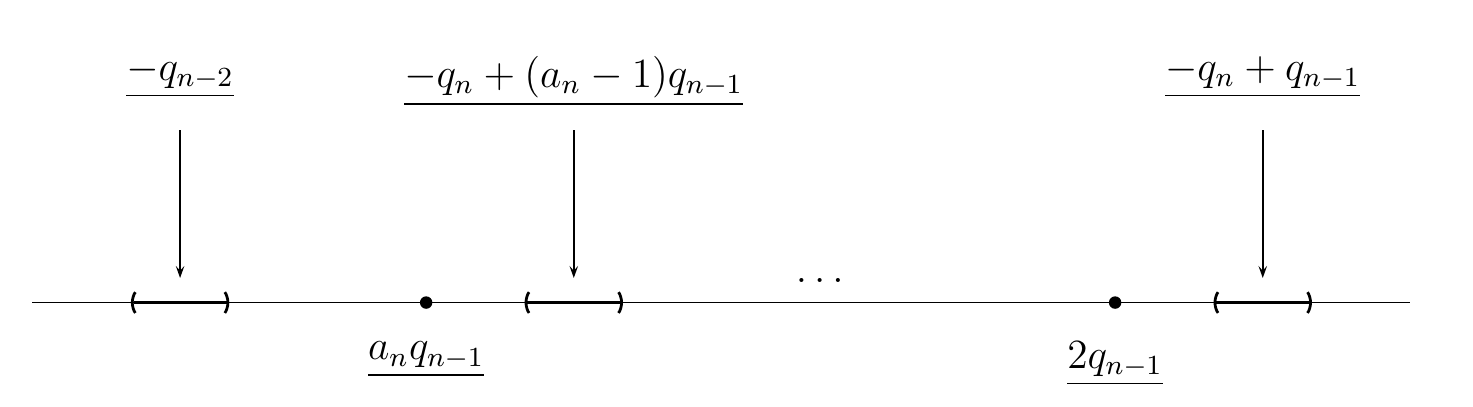}
		
		\end{center}
\end{figure}
 \begin{figure}[h]
		\begin{center}
		\includegraphics[width=15.5cm]{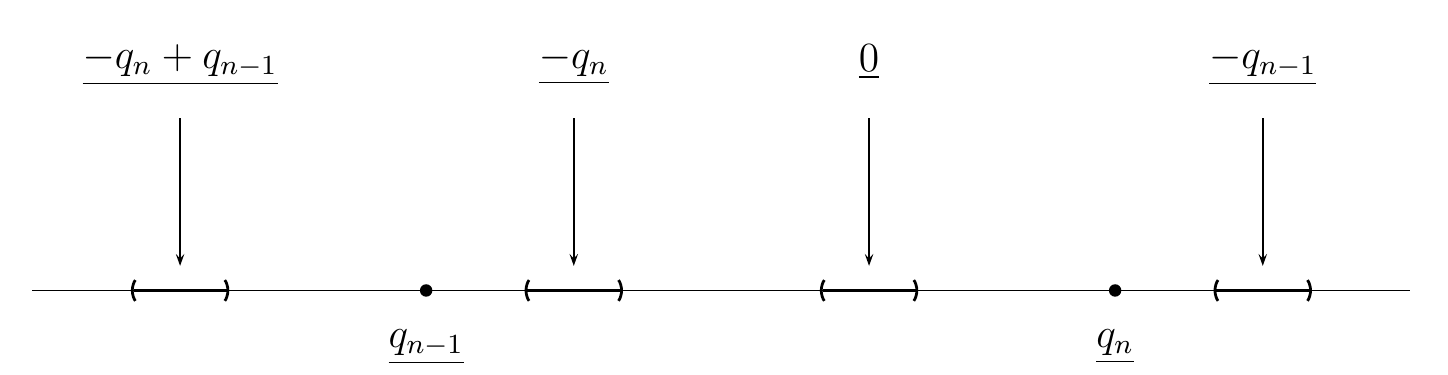}
		\caption{Structure of the dynamical partition $\mathscr P_{n-1}$ for $n$ even and $a_n>1$.}
		\label{fig:primera}
		\end{center}
\end{figure}

We state a standard fact and few results from \cite{5aut} which will be used frequently in the paper.
\begin{fact}\label{F1}
Let $f\in\L$ and let $x,y,z$ be three points of the circle with $y$ between $x$ and $z$ such that, of the three, the point $x$ is closest to the flat interval. If $f$ is a diffeomorphism on $(x,z)$, the following inequality holds:
\begin{displaymath}
\frac{|f(z)-f(y)|}{|f(z)-f(x)|}\leq K\frac{|z-y|}{|z-x|},
\end{displaymath}
where $K$ is a positive uniform constant.

\end{fact}

\begin{proposition}\label{figo1}
There exists a constant $C > 0$, such that for all $n \in \N$ and $m \in \N$, if
$J = f^{-m}(U)$ is a preimage of the flat interval $U$ which belongs to the dynamical partition $\mathscr P_n$ and $I$ is one of the two gaps adjacent to $J$, then:
$$\frac{\left|J\right|}{\left|I\right|}\geq C$$
\end{proposition}
\begin{corollary}\label{cor:trouzero}
The lengths of gaps of the dynamical partition $\mathscr{P}_{n}$ tend to zero at least exponentially fast when $n\to\infty$.
\end{corollary}
The proofs of Proposition \ref{figo1} and Corollary \ref{cor:trouzero} can be found in \cite{5aut}, pag. 606-607.
\paragraph{Standing assumption.} In the following we will always work with functions in $\L$ which have critical exponent $\ell>1$ and irrational rotation number.

\section{Proof of Theorem \ref{sw1}}
The first claim of Theorem \ref{sw1} is proved in \cite{my} under the additional assumption that the rotation number is of bounded type. For the general case, the details are provided in the Appendix. We proceed now with the proof of the second claim of Theorem \ref{sw1}.

\subsection{Some Technical Lemmas}

We present some technical lemmas which we need for the proof of the main theorem.
\paragraph{Standing assumption.} Because of the symmetry of the functions in $\L$ we always assume that $n \in \N$ is even. The case $n \in \N$ odd is completely analogous.

\begin{lemma}\label{prim}
There exists a constant $K > 0$, such that the fraction
\begin{equation*}
\frac{|(\underline{2q_{n+1}},\underline{q_{n+1}})|}{|(\underline{2q_{n+1}},\underline{0})|} > K > 0.
\end{equation*}
\end{lemma}

\begin{proof}
See Lemma $1.2$ in \cite{5aut}.
\end{proof}

\begin{lemma}\label{altro}
There exists a constant $K>0$, such that for $n$ large enough the fraction
\begin{equation*}
\frac{|\underline{-q_{n}-q_{n+1}}|}{|[\underline{-q_{n}-q_{n+1}},\underline{0})|}>K>0.
\end{equation*}
\end{lemma}
\begin{proof}
The reader can keep in mind Figure \ref{fig:zoomsurq}.\\

By Fact \ref{F1} there exists a constant $K_1 > 0$ such that

\begin{eqnarray*}
\frac{|\underline{-q_{n}-q_{n+1}}|}{|[\underline{-q_{n}-q_{n+1}},\underline{0})|} &\geq& K_1\frac{|\underline{-q_{n}-q_{n+1}+1}|}{|[\underline{-q_{n}-q_{n+1}+1},\underline{1})|}\geq \\
&\geq& K_1\frac{|\underline{-q_{n}-q_{n+1}+1}|}{|[\underline{-q_{n}-q_{n+1}+1},\underline{-q_{n+1}+1})|}\frac{|\underline{-q_{n+1}+1}|}{|(\underline{-q_{n}-q_{n+1}+1},\underline{-q_{n+1}+1}]|}.
\label{for}
\end{eqnarray*}

\par
We apply $f^{q_{n+1}-1}$. By the properties of distortion of cross-ratio $\mathbf{Cr}$, there exists a positive constant $K_2$ such that:
\begin{eqnarray}\label{for1}
\frac{|\underline{-q_{n}-q_{n+1}}|}{|[\underline{-q_{n}-q_{n+1}},\underline{0})|} &\geq&K_1 K_2\frac{|\underline{-q_{n}}|}{|[\underline{-q_{n}},\underline{0})|}\frac{|\underline{0}|}{|(\underline{-q_{n}},\underline{0}]|}\geq K_1 K_2\frac{|\underline{-q_{n}}|}{|[\underline{-q_{n}},\underline{0})|}.
\end{eqnarray}
\par
For $n$ large enough we can discard the intervals containing $\underline{0}$ and using Proposition \ref{figo1} the proof is complete.
\end{proof}
 \begin{figure}[h]
		\begin{center}
		\includegraphics[width=15.5cm]{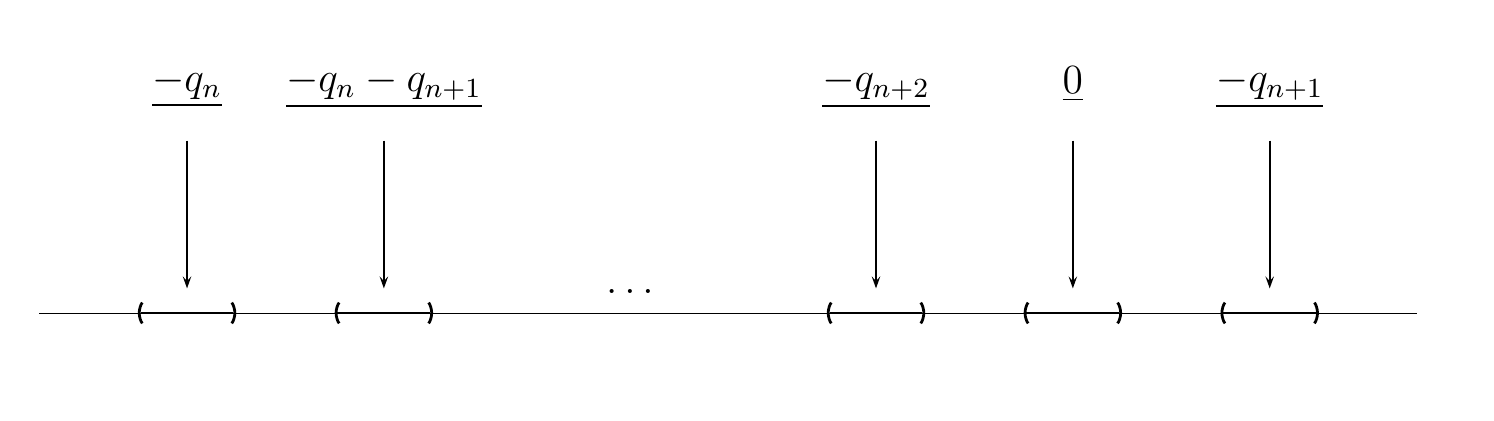}
		\caption{}
             \label{fig:zoomsurq}
		\end{center}
\end{figure}
%
%
%
%
%
%
%
%
%
%
%
%

\begin{lemma}\label{L1}
There exists a constant $K>0$ such that, for $n$ large enough,
\begin{equation*}
\frac{ |\underline{-q_{n-1}-q_n+1}|}{|(\underline{-q_n+1},\underline{1})|}\geq K.
 \end{equation*}
\end{lemma}

\begin{proof}
The reader can keep in mind Figure \ref{fig:zoomsurq1}.\\

We have:
\begin{eqnarray}
\frac{|\underline{-q_{n-1}-q_n+1}|}{|(\underline{-q_n+1},\underline{1})|} &\geq& \frac{|\underline{-q_{n-1}-q_n+1}|}{|(\underline{-q_n+1},\underline{-q_{n-1}-q_n+1}]|} \\
&\geq& \frac{|\underline{-q_{n}+1}|}{|[\underline{-q_n+1},\underline{-q_{n-1}-q_n+1})|}\frac{|\underline{-q_{n-1}-q_n+1}|}{|(\underline{-q_n+1},\underline{-q_{n-1}-q_n+1}]|}. \nonumber
\label{zuc}
\end{eqnarray}
After $(q_{n}-1)$ iterates, by the distortion properties of cross-ratio $\mathbf{Cr}$, there exists a constant $K_1>0$ such that:
\begin{equation*}
\frac{|\underline{-q_{n-1}-q_n+1}|}{|(\underline{-q_n+1},\underline{1})|} \geq K_1\frac{|\underline{0}|}{|[\underline{0},\underline{-q_{n-1}})|}\frac{|\underline{-q_{n-1}}|}{|(\underline{0},\underline{-q_{n-1}}]|}
\end{equation*}
which, for $n$ large enough, is bounded below by a positive constant (Proposition \ref{figo1}).
\end{proof}
 \begin{figure}[h]
		\begin{center}
		\includegraphics[width=15.5cm]{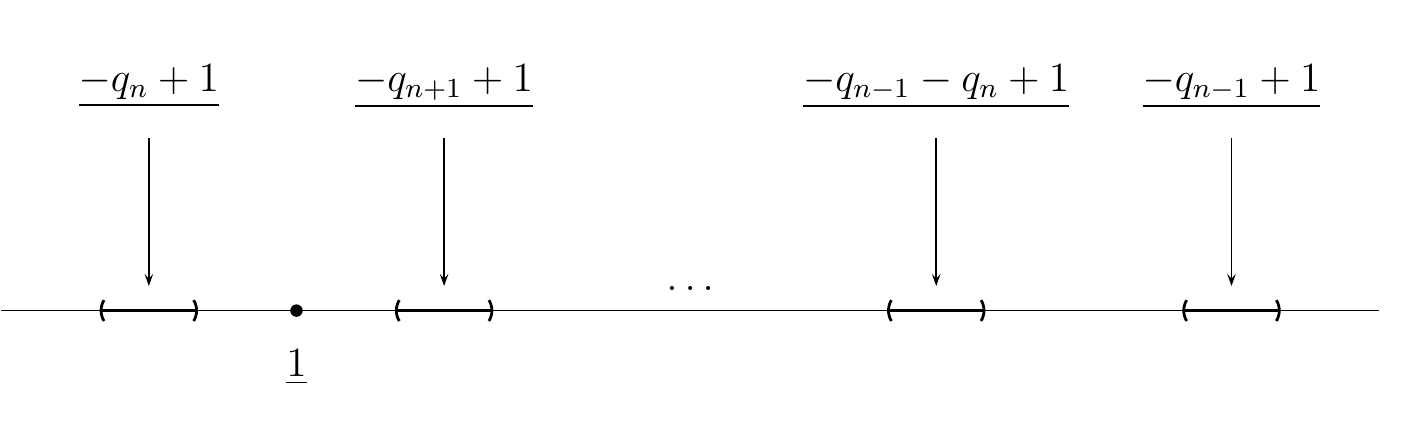}
		\caption{}
            \label{fig:zoomsurq1}
		\end{center}
\end{figure}
%
%
%
%
%
%
%
%

For all $n$ and for all $i\in\{0,\dots, a_{n+2}-1\}$ we define (see Figure \ref{fig:bg}):
 \begin{figure}[h]
		\begin{center}
		\includegraphics[width=15.5cm]{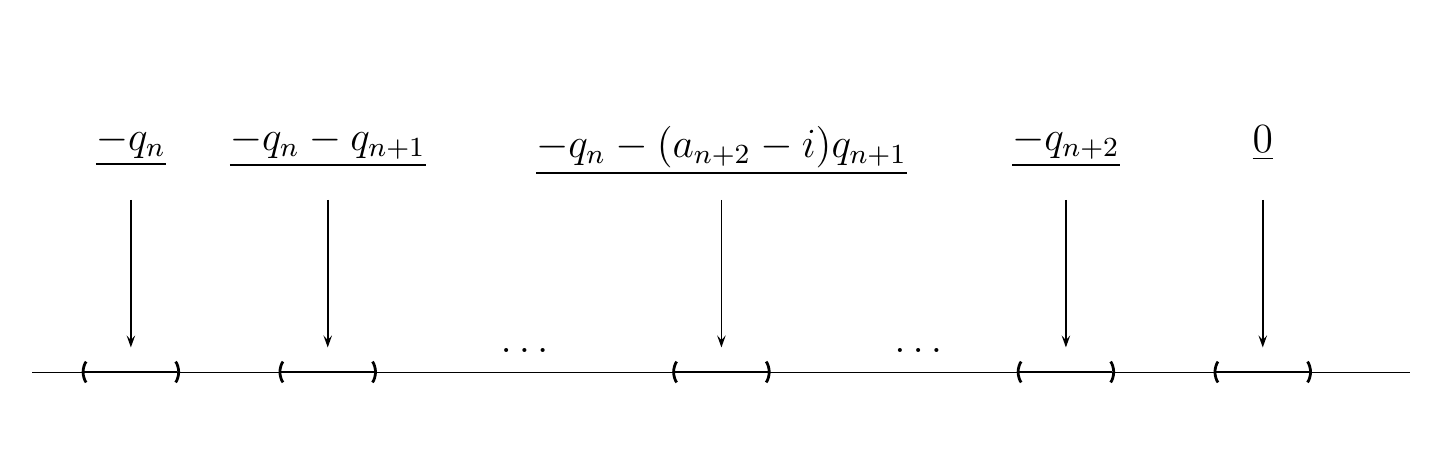}
		\caption{}
             \label{fig:bg}
		\end{center}
\end{figure}
%
%
%
%
%
%
%
%
%
%
%
%

\begin{equation*}
\beta_n(i)=\frac{|(\underline{-q_n-(a_{n+2}-i)q_{n+1}},\underline{0})|}{|[\underline{-q_n-(a_{n+2}-i)q_{n+1}},\underline{0})|}
\end{equation*}
and
\begin{equation*}
\gamma_n(i)=\frac{|[\underline{-q_n-(a_{n+2}-i)q_{n+1}},\underline{0})|}{|(\underline{-q_n-(a_{n+2}-(i+1))q_{n+1}},\underline{0})|}
\end{equation*}
and we prove the following lemma:

\begin{lemma}\label{beta}
There exists a constant $K > 0$, such that for all
$ i \in \{ 1, \dots, a_{n+2} - 2 \}$, we have:
\begin{equation*}
(\beta_n(i))^\ell\geq K\beta_n(i+1).
\end{equation*}
\end{lemma}

We observe that this lemma makes sense under the assumption that $a_{n+2} \notin \{1,2\}$.

\begin{proof}
We apply $f$ to the intervals defining $\beta_n(i)$ and we obtain, for large $n$,
\begin{equation*}
\beta_n(i)^\ell=\frac{|(\underline{-q_n-(a_{n+2}-i)q_{n+1}+1},\underline{1})|}{|[\underline{-q_n-(a_{n+2}-i)q_{n+1}+1},\underline{1})|}.
\end{equation*}
\par
For all $i\in\{1,\dots, a_{n+2}-2\}$  we apply Proposition \ref{Koebe liv} to

\begin{itemize}
\item[-]$T=[\underline{-q_{n}-q_{n+1}+1}, \underline{-q_{n+1}+1}]$,
\item[-]$J=(\underline{-q_{n}-q_{n+1}+1}, \underline{-q_{n+1}+1})$,
\item[-]$f^{q_{n+1}-1}$.
\end{itemize}

We notice that the hypotheses are satisfied:

\begin{itemize}
\item[-] $f^{q_{n+1}-1}$ is a diffeomorphism on $T$,
\item[-] the intervals $f^j(J)$ for $j\in\left\{1,\dots,q_{n+1}-2\right\}$ are disjoint,
\item[-] by Proposition \ref{figo1},  for $n$ large enough, there exists a positive  constant $K_1$ such that
\begin{displaymath}
\left|f^{q_{n+1}-1}(J)\right|=\left|(\underline{-q_n},\underline{0})\right|\leq K_1 \left|\underline{-q_{n}}\right|=K_1 \dist(f^{q_{n+1}-1}(J),\partial f^{q_{n+1}-1}(T)).
\end{displaymath}
\end{itemize}
Then we find a uniform constant $K_2>0$ such that:
\begin{eqnarray}
\beta_n(i)^\ell&=&\frac{|(\underline{-q_n-(a_{n+2}-i)q_{n+1}+1},\underline{1})|}{|[\underline{-q_n-(a_{n+2}-i)q_{n+1}+1},\underline{1})|}\\&\geq& K_2\frac{|(\underline{-q_n-(a_{n+2}-(i+1))q_{n+1}},\underline{q_{n+1}})|}{|[\underline{-q_n-(a_{n+2}-(i+1))q_{n+1}},\underline{q_{n+1}})|}.
\label{beaux}
\end{eqnarray}
Since the numerator of (\ref{beaux}) contains the interval $(\underline{2q_{n+1}},\underline{q_{n+1}})$, for Lemma \ref{prim} we can conclude that $(\beta_n(i))^\ell$ is greater than a positive constant multiplied by $\beta_n(i+1)$.
\end{proof}

\begin{lemma}\label{gamma}
There exist two constant $K_1 > 0$ and $K_2 > 0$ such that:
\begin{enumerate}
\item for all $0\leq i\leq a_{n+2}-2$, $(\gamma_n(i))^\ell\geq K_1\gamma_n(i+1)$,
\item $\gamma_n( a_{n+2}-1)\geq K_2$.
\end{enumerate}
\end{lemma}
\begin{proof}
In order to prove point $(2)$ it is sufficient to observe that:
\begin{equation*}
\gamma_n( a_{n+2}-1)=\frac{|[\underline{-q_n-q_{n+1}},\underline{0})|}{|(\underline{-q_n},\underline{0})|}
\end{equation*}
which is greater than a positive uniform constant (see Proposition \ref{figo1}).
\par
In order to prove point $(1)$ we first apply $f$ to intervals defining $\gamma_n(i)$ and then Proposition \ref{Koebe liv} to
\begin{itemize}
\item[-]$T=[\underline{-q_{n}-q_{n+1}+1}, \underline{-q_{n+1}+1}]$,
\item[-]$J=(\underline{-q_{n}-q_{n+1}+1}, \underline{-q_{n+1}+1})$,
\item[-]$f^{q_{n+1}-1}$.
\end{itemize}
Like in Lemma \ref{beta}, the hypotheses are satisfied, so there exists a constant $K_1>0$ such that, for all $i$, $(\gamma_n(i))^\ell\geq K_1\gamma_n(i+1)$.
\end{proof}
In the following, in order to simplify notation, we note $\beta_n=\beta_n(a_{n+2}-1)$ and $\gamma_n=\gamma_n(a_{n+2}-1)$.

\subsection{The Central Part of the Proof}\label{CPP}
We recall that
\begin{displaymath}
\tau_n=\frac{|(\underline{0},\underline{q_n})|}{|(\underline{0},\underline{q_{n-2}})|},
\end{displaymath}
\begin{displaymath}
\alpha_n=\frac{|(\underline{-q_n},\underline{0})|}{|[\underline{-q_n},\underline{0})|}
\end{displaymath}
and we introduce a new parameter which measures the relative size of $\alpha_n$ and $\tau_n$,
\begin{equation*}
k_n=\frac{|(\underline{0},\underline{q_n})|}{|(\underline{0},\underline{-q_{n-1}})|}.
\end{equation*}

\begin{remark}\label{finpr}
We recall that the point $\underline{q_{n-2}}$ is situated in the gap between $\underline{-q_{n-1}}$ and $\underline{-q_{n-1}+q_{n-2}}$ of the dynamical partition $\mathscr{P}_{n-2}$.
\par
Then, by Proposition \ref{figo1}, there exists a constant $K>0$ such that
\begin{equation*}
 k_n\geq\frac{\tau_n}{\alpha_{n-1}}\geq K k_n.
 \end{equation*}
 \par
Finally, $\frac{\tau_n}{\alpha_{n-1}}$ and $k_n$ are comparable.
\end{remark}
To complete the proof of the second claim of Theorem \ref{sw1} it is necessary to find a bound for the sequence $(\alpha_n)_{n \in \N}$ and $(k_n)_{n \in \N}$. For this reason we prove the following propositions.

\begin{proposition}\label{n7}
There exists a positive constant $K$ such that, for $n$ large enough
\begin{equation*}
k_n\geq K\beta_{n-1}^{\left(\frac{1-\ell^{-a_{n+1}-1}}{1-\ell^{-1}}\right)}.
\end{equation*}

\end{proposition}
\begin{proof}
By Proposition \ref{figo1} there exists a uniform constant $K_1>0$ such that
\begin{equation}\label{eq:pr1}
k_n\geq K_1\frac{|(\underline{0},\underline{-q_{n-1}-(a_{n+1}-1)q_n})|}{|(\underline{0},\underline{-q_{n-1}})|}.
\end{equation}
\par

For all $i\in\{1,\dots, a_{n+1}-1\}$ and for all $j\in\{2,\dots, a_{n+1}-2\}$, we multiply and divide alternatively by $|(\underline{0},\underline{-q_{n-1}-(a_{n+1}-i)q_n}]|$ and by \\$|(\underline{0},\underline{-q_{n-1}-(a_{n+1}-j)q_n})|$ to obtain that:
\begin{equation*}
k_n\geq K_1\beta_{n-1}(1)\gamma_{n-1}(2)\beta_{n-1}(2)\dots\gamma_{n-1}(a_{n+1}-2)\beta_{n-1}(a_{n+1}-1)\frac{|(\underline{0},\underline{-q_{n-1}-q_n}]|}{|(\underline{0},\underline{-q_{n-1}})|}.
\end{equation*}
\par
By Proposition \ref{figo1}, there exists a positive constant $K_2$ such that
\[
\frac{|(\underline{0},\underline{-q_{n-1}-q_n}]|}{|(\underline{0},\underline{-q_{n-1}})|}> K_2.
\]
We apply Lemma \ref{beta}, Lemma \ref{gamma} and we have:
\begin{equation}\label{eq:xa1}
k_n\geq K_3\beta_{n-1}\beta_{n-1}^{\frac{1}{\ell}}\dots\beta_{n-1}^{\left( \frac{1}{\ell^{a_{n+1}-2}}\right)}=K_3\beta_{n-1}^{\left(\frac{1-\ell^{-a_{n+1}+1}}{1-\ell^{-1}}\right)}.
\end{equation}
where $K_3$ is a positive constant.\\

It remains to study the case $a_{n+1}=1$ and $a_{n+1}=2$ for which we can't use Lemma \ref{beta}.
\\

We assume that $a_{n+1}=1$. In this case the gap on the right of  $\underline{-q_{n+1}}$ is $(\underline{-q_{n+1}},\underline{-q_{n-1}})$ and then by Proposition \ref{figo1} there exists $K_4>0$ such that $k_n\geq K_4$.
\par

We find the same inequality that (\ref{eq:xa1}) in the specific case $a_{n+1}=1$.\\

If $a_{n+1}=2$ we proceed like in the general case up to obtain (like in   (\ref{eq:pr1})) that:
\begin{equation*}
k_n\geq K_1\frac{|(\underline{0},\underline{-q_{n-1}-q_n})|}{|(\underline{0},\underline{-q_{n-1}})|}
\end{equation*}
which by Proposition \ref{figo1} is greater than a uniform positive constant multiplied by $\beta_{n-1}$.
\\

The proof of the proposition is then complete.
\end{proof}

\begin{proposition}\label{n1}
There exists a constant $K>0$ such that, for $n$ large enough
\begin{equation*}
\alpha_n\geq K\beta_{n-1}^{\ell\left(\frac{1-\ell^{-a_{n+1}}}{\ell-1}\right)}\beta_{n-2}^{\ell^{-a_{n}+1}}.
\end{equation*}
\end{proposition}

\begin{proof}
If $n$ is large enough, $(\alpha_n)^\ell$ is equal to the fraction
\begin{equation*}
\frac{|(\underline{-q_n+1},\underline{1})|}{|[\underline{-q_n+1},\underline{1})|}
\end{equation*}
which is greater than the product of the followings three fractions:
\begin{eqnarray*}
\xi_1&=&\frac{|(\underline{-q_n+1},\underline{1})|}{|(\underline{-q_n+1},\underline{-q_{n-1}-q_{n}+1})|}, \\ \xi_2&=&\frac{|(\underline{-q_n+1},\underline{-q_{n-1}-q_{n}+1})|}{|(\underline{-q_n+1},\underline{-q_{n-1}+1})|}, \\
\xi_3&=&\frac{|(\underline{-q_n+1},\underline{-q_{n-1}+1})|}{|[\underline{-q_n+1},\underline{-q_{n-1}+1})|}.
\end{eqnarray*}

We focus on each fraction separately.
\begin{itemize}
\item[$1^{st}$ step.] We prove that $\xi_1\geq K_1\beta_{n-1}^{\left(\frac{1-\ell^{-a_{n+1}-1}}{1-\ell^{-1}}\right)}$.\\

We apply Proposition \ref{Koebe liv} to
\begin{itemize}
\item[-]$T=[\underline{-q_{n}+1}, \underline{-q_{n}-q_{n-1}+1}]$,
\item[-]$J=(\underline{-q_{n}+1}, \underline{-q_{n}-q_{n-1}+1})$,
\item[-]$f^{q_{n}-1}$.
\end{itemize}
Like in the previous lemmas the hypothesis are satisfied, so there exists a positive constant $C_1$ such that:
\begin{equation}\label{eq:pr}
\xi_1\geq C_1\frac{|(\underline{0},\underline{q_n})|}{|(\underline{0},\underline{-q_{n-1}})|}.
\end{equation}
\par
We now apply Proposition \ref{n7} and we find the wanted estimate.
\\

\item[$2^{nd}$ step.] We prove that $\xi_2\geq K_2\beta_{n-1}^\ell$\\

By Proposition \ref{figo1} and Lemma \ref{L1} we have two positives constants $C_2 $ and $C_3 $ such that
\begin{equation}\label{eq:xe}
\xi_2\geq C_2\frac{|(\underline{1},\underline{-q_{n-1}-q_{n}+1})|}{|(\underline{-q_n+1},\underline{-q_{n-1}-q_{n}+1}]|}\geq C_2 C_3 \frac{|(\underline{1},\underline{-q_{n-1}-q_{n}+1})|}{|(\underline{1},\underline{-q_{n-1}-q_{n}+1}]|}
\end{equation}
and this last fraction is exactly $\beta_{n-1}^\ell$.\\

\item[$3^{rd}$ step.] We prove that $\xi_3\geq K_3\beta_{n-2}^{\ell^{-a_n+2}}$. \\

We apply Lemma \ref{beta} and Proposition \ref{Koebe liv} to
\begin{itemize}
\item[-]$T=[\underline{-q_{n-2}-q_{n-1}+1}, \underline{-q_{n-1}+1}]$,
\item[-]$J=(\underline{-q_{n-2}-q_{n-1}+1}, \underline{-q_{n-1}+1})$,
\item[-]$f^{q_{n-1}-1}$.
\end{itemize}
and we find, under the assumption that $a_n\notin \{1,2\}$, two positive constants $C_4 > 0$ and $C_5 > 0$ such that:
\begin{equation}\label{eq:xi}
\xi_3\geq C_4\frac{|(\underline{-q_n+q_{n-1}},\underline{0})|}{|[\underline{-q_n+q_{n-1}},\underline{0})|}\geq C_4\beta_{n-2}(1)\geq C_4 C_5\beta_{n-2}^{\ell^{-a_n+2}}.
\end{equation}
\\

If $a_n=1$, we obtain exactly the same estimate, in fact, in this case $\underline{-q_n+1}=\underline{-q_{n-2}-q_{n-1}+1}$ and by Lemma \ref{L1} there exists a constant $C_6 > 0$ such that
\begin{equation*}
\xi_3\geq C_{6}\beta_{n-2}^\ell.
\end{equation*}
\\
If $a_n=2$, we proceed like the general case until the first inequality (\ref{eq:xi}). Now it is sufficient to observe that, in this case,
$$\frac{|(\underline{-q_n+q_{n-1}},\underline{0})|}{|[\underline{-q_n+q_{n-1}},\underline{0})|}=\beta_{n-2}.$$
\end{itemize}

Finally, using the estimates obtained for $\xi_1$, $\xi_2$ and $\xi_3$ we have
\begin{equation*}
\alpha_n^\ell\geq K\beta_{n-1}^{\left(\frac{1-\ell^{-a_{n+1}+1}}{1-\ell^{-1}}\right)}\beta_{n-1}^\ell\beta_{n-2}^{\ell^{-a_n+2}}
\end{equation*}
therefore
\begin{equation*}
\alpha_n\geq K^{\frac{1}{\ell}}\beta_{n-1}^{\ell\left(\frac{1-\ell^{-a_{n+1}}}{\ell-1}\right)}\beta_{n-2}^{\ell^{-a_{n}+1}}.
\end{equation*}

\end{proof}
\par
By Remark \ref{finpr}, Proposition \ref{n7} and Proposition \ref{n1}, in order to find a lower bound for $\tau_n$, it is necessary to study the sequence $\beta_n$. Hence the following propositions:
\begin{proposition}\label{n2}
There exists a positive constant $K$ such that, for $n$ large enough
\begin{equation*}
\beta_n\geq K\beta_{n-1}^{\frac{1}{\ell}}\alpha_{n}^{\frac{1}{\ell}}
\end{equation*}
\end{proposition}

\begin{proof}
We start by applying Proposition \ref{Koebe liv} to

\begin{itemize}
\item[-]$T=[\underline{-q_{n}+1}, \underline{-q_{n-1}-q_n+1}]$,
\item[-]$J=(\underline{-q_{n}+1}, \underline{-q_{n-1}-q_n+1})$,
\item[-]$f^{q_{n}-1}$.
\end{itemize}

Then, for $n$ large enough, there exists a constant $C_1 > 0$ such that
\begin{equation*}
\beta_n^{\ell}=\frac{|(\underline{-q_{n}-q_{n+1}+1}, \underline{1})|}{|[\underline{-q_{n}-q_{n+1}+1}, \underline{1})|}\geq C_1 \frac{|(\underline{-q_{n+1}}, \underline{q_n})|}{|[\underline{-q_{n+1}}, \underline{q_n})|}.
\end{equation*}

Multiplying and dividing by $|(\underline{-q_{n+1}}, \underline{-q_{n+1}+q_{n}})|$, we find that $\beta_n^{\ell}$ is greater than (up to a constant) the product of the following two quantities:

\begin{equation*}
\eta_1=\frac{|(\underline{-q_{n+1}}, \underline{q_n})|}{|(\underline{-q_{n+1}}, \underline{-q_{n+1}+q_{n}})|},  \eta_2=\frac{|(\underline{-q_{n+1}}, \underline{-q_{n+1}+q_{n}})|}{|[\underline{-q_{n+1}}, \underline{q_n})|}.
\end{equation*}

We focus our attention on these two quantities.\\

\begin{itemize}
\item[$1^{st}$ step.] We prove that $\eta_1\geq K_1$.\\

We use Proposition \ref{Koebe liv} with
\begin{itemize}
\item[-]$T=[\underline{-q_{n+1}}, \underline{-q_{n+1}+q_{n}}]$,
\item[-]$J=(\underline{-q_{n+1}}, \underline{-q_{n+1}+q_{n}})$,
\item[-]$f^{q_{n+1}-q_{n}}$
\end{itemize}

and we find a constant $C_2 > 0$ such that

\begin{eqnarray*}
\eta_1&\geq& C_2\frac{|(\underline{-q_{n}}, \underline{q_{n+1}})|}{|(\underline{-q_{n}}, \underline{0})|}\geq C_2\frac{|(\underline{-q_{n}}, \underline{-q_n-q_{n+1}}]|}{|(\underline{-q_{n}}, \underline{0})|}\\&\geq& C_2 C_3 \frac{|\underline{-q_{n}-q_{n+1}}|}{|[\underline{-q_{n}-q_{n+1}},\underline{0})|}\geq  C_2 C_3 C_4.
\end{eqnarray*}

We observe that $C_3$ comes from Proposition \ref{figo1} and $C_4$ from Lemma \ref{altro}. \\

\item[$2^{nd}$ step.] We show that $\eta_2\geq K_2\beta_{n-1}\alpha_{n} $. \\

By Lemma \ref{prim} there exists a constant $C_5 > 0$ such that:

\begin{eqnarray*}
\eta_2&\geq& C_5\frac{|(\underline{-q_{n+1}}, \underline{-q_{n+1}+q_n})|}{|(\underline{-q_{n+1}}, \underline{-q_{n+1}+q_{n}}]|}\\&\geq& C_5\frac{|(\underline{-q_{n+1}}, \underline{-q_{n+1}+q_n})|}{|(\underline{-q_{n+1}}, \underline{-q_{n+1}+q_{n}}]|}\frac{|(\underline{-q_{n+1}+q_n}, \underline{-q_{n+1}+2q_n}]|}{|[\underline{-q_{n+1}+q_n}, \underline{-q_{n+1}+2q_{n}}]|}.
\end{eqnarray*}

After $(a_{n+1}-2)q_n$ iterates, by the properties of distortion of cross-ratio $\mathbf{Cr}$ we have that:

\begin{eqnarray*}
\eta_2 &\geq& C_6\beta_{n-1}\frac{|(\underline{-q_{n-1}-2q_n}, \underline{-q_{n-1}-q_n})|}{|(\underline{0}, \underline{-q_{n-1}-q_{n}})|} \\
       &\geq& C_6\beta_{n-1}\frac{|\underline{-q_{n}}|}{|[\underline{-q_n}, \underline{-q_{n-1}-2q_{n}}]|}\frac{|(\underline{-q_{n-1}-2q_n}, \underline{-q_{n-1}-q_n})|}{|(\underline{-q_n}, \underline{-q_{n-1}-q_{n}})|}.
\end{eqnarray*}

Applying the properties of distortion of cross-ratio $\mathbf{Cr}$ after $q_n$ iterates and by Lemma \ref{altro} and Proposition \ref{figo1} we find that, for $n$ large enough:

\begin{equation*}
\eta_2\geq C_7\beta_{n-1}\frac{|(\underline{-q_{n-1}-q_n}, \underline{-q_{n-1}})|}{|\underline{-q_{n-1}-q_{n}}|}\geq C_7\beta_{n-1}\frac{|(\underline{-q_{n-1}-q_n}, \underline{-q_{n-1}})|}{|[\underline{-q_{n-1}-q_n}, \underline{-q_{n-1}})|}.
\end{equation*}

It remains to find a bound for $\frac{|(\underline{-q_{n-1}-q_n}, \underline{-q_{n-1}})|}{|[\underline{-q_{n-1}-q_n}, \underline{-q_{n-1}})|}$. First we apply Fact \ref{F1} and   Proposition \ref{Koebe liv} to

\begin{itemize}
\item[-]$T=[\underline{-q_{n-1}-q_{n-2}+1}, \underline{-q_{n-1}+1}]$,
\item[-]$J=(\underline{-q_{n-1}-q_{n-2}+1}, \underline{-q_{n-1}+1})$,
\item[-]$f^{q_{n-1}-1}$.
\end{itemize}

to get two constants $C_8 > 0$ and $C_9 > 0$ such that:

\begin{eqnarray*}
\frac{|(\underline{-q_{n-1}-q_n}, \underline{-q_{n-1}})|}{|[\underline{-q_{n-1}-q_n}, \underline{-q_{n-1}})|} &\geq& C_8\frac{|(\underline{-q_{n-1}-q_n+1}, \underline{-q_{n-1}+1})|}{|[\underline{-q_{n-1}-q_n+1}, \underline{-q_{n-1}+1})|} \\
&\geq& C_9\frac{|(\underline{-q_n}, \underline{0})|}{|[\underline{-q_n}, \underline{0})|}\geq C_9\alpha_n.
\end{eqnarray*}

\end{itemize}
In order to conclude the proof, it is sufficient put together the bounds found for $\eta_1$ and $\eta_2$.
\end{proof}

Propositions \ref{n1} and \ref{n2} give us the following important inequality:

\begin{theorem}\label{n3}
There exists a positive constant $K$ such that, for $n$ large enough
\begin{equation*}
\beta_n\geq K\beta_{n-1}^{\left(\frac{1}{\ell}+\frac{1-\ell^{-a_{n+1}}}{\ell-1}\right)}\beta_{n-2}^{\ell^{-a_{n}}}.
\end{equation*}
\end{theorem}

To complete the proof of the second claim of Theorem \ref{sw1}, by Remark \ref{finpr}, Proposition \ref{n7} and Proposition \ref{n1} it is sufficient to prove that the sequence $(\beta_n)_{n \in \N}$ is bounded away from zero for $\ell\geq 3$. It remains to analyze the recurrence of the sequence $(\beta_n)_{n \in \N}$.

\paragraph{Analysis of the Recurrence.} We define for all $n \in \N$ the quantity
\begin{equation*}
\nu_n= - \ln \beta_n
\end{equation*}

Theorem \ref{n3} implies that there exists a constant $K_1 > 0$ such that:
\begin{equation}\label{chis}
\nu_n-\left({\frac{1}{\ell}+\frac{1-\ell^{-a_{n+1}}}{\ell-1}}\right)\nu_{n-1}-{\ell^{-a_n}}\nu_{n-2}\leq K_1.
\end{equation}
We prove that the sequence $(\nu_n)_{n \in \N}$ is bounded. In order to do this we start to consider the sequence of vectors $(v_n)_{n \in \N}$:
\begin{displaymath}
v_n={\nu_n\choose\nu_{n-1} },
\end{displaymath}
the sequence of matrices $(A_{\ell}(n))_{n \in \N}$:
\begin{displaymath}
A_{\ell}(n) = \left( \begin{matrix}{\frac{1}{\ell}+\frac{1-\ell^{-a_{n+1}}}{\ell-1}} &{\ell^{-a_n}} \\
1 & 0
\end{matrix} \right)
\end{displaymath}
and the vector
\begin{displaymath}
k={K_1\choose 0 }.
\end{displaymath}

Now we can write (\ref{chis}) in the form

\begin{equation}\label{vit}
v_n\leq A_{\ell}(n)A_{\ell}(n-1)\dots A_{\ell}(2)v_1+\left(\sum_{i=2}^{n-1}A_{\ell}(n-1)A_{\ell}(n-2)\dots A_{\ell}(i)\right)k
\end{equation}

where the inequality must be read component-wise.

\par
To prove that the sequence $(v_n)_{n \in \N}$ is bounded and then prove the second claim of Theorem \ref{sw1}, it is necessary to study for all $n \in \N$ and $2 \leq i < n$ each product,
$$ \prod_{j = i}^n A_{\ell}(j) $$
\par
In particular we will estimate $\|\prod_{j = i}^n A_{\ell}(j)\|_{\infty}$.
\par
\begin{remark}
Recall that if $A = (a_{i,j})_{1 \leq i, j \leq n}$ is a matrix, then $$\|A\|_{\infty}=\max_{1\leq i\leq n}\sum_{1\leq j\leq n}|a_{i,j}|$$ is an operator norm. We observe in fact that $$\|A\|_{\infty}=\max_{v\in\R^n\setminus\{0\}}\frac{\|Av\|_{\infty}}{\|v\|_{\infty}}$$ where $\|v\|_{\infty}=\max_{1\leq i\leq n}|v_i|$ if $v=(v_1,\dots v_n)$.\\
Moreover if $A = (a_{i,j})_{1 \leq i, j \leq n}$ and $B = (b_{i,j})_{1 \leq i, j \leq n}$ and for all $(i,j)\in \{1,\ldots,n\}^2$, $a_{ij}\leq b_{ij}$, we use the shortcut notation $A \leq B$. In this case we have that $\|A\|_{\infty}\leq \|B\|_{\infty}$.
\end{remark}

We fix now $n$ and $i$ such that $2\leq i < n$, and for all $i\leq j\leq n$ we denote $b_j = \ell^{-a_{j+1}}$, hence
\begin{displaymath}
A_{\ell}(j)=\left( \begin{matrix}{\frac{1}{\ell}+\frac{1-b_j}{\ell-1}} &{b_{j-1}} \\
1 & 0
\end{matrix} \right).
\end{displaymath}
We observe that the sequence of reals positives numbers $(b_j)_{j \in \N}$ is bounded by $1/\ell$.\\

We fix an integer $M>1$. Then for all $j\in\lbrace{i,n-1\rbrace}$ we can have three different cases:
\begin{enumerate}
\item  $a_{j+1} <M$ and $a_j<M$,
\item $a_{j+1} \geq M$,
\item $a_j\geq M$.
\end{enumerate}
In the first case we denote $B:=A_{\ell}(j)$, in the second one $U_1:=A_{\ell}(j)$ and in the third one $U_2:=A_{\ell}(j)$ .
\par
We observe that, by point $(1)$ we cannot find products of the type $BU_1B$ or $BU_2B$ , because the two matrices $B$ affect the central matrix.\\

We fix now $\ell\geq 3$, $j\in\lbrace{i,n-1\rbrace}$ and we consider the different combinations of matrices $B$, $U_1$ and $U_2$ which we can have. \\

\begin{itemize}

\item We start by considering a product of matrices of type $BBB\cdots BB$.
\par
We observe that if $a_{j+1} <M$ and $a_j<M$ then by estimations in \cite{5aut}
\begin{equation*}
B=A_{\ell}(j)=\left( \begin{matrix}{\frac{1-b_j}{\ell-1}} &{b_{j-1}} \\
1 & 0
\end{matrix} \right)
\end{equation*}

\par
Since $\ell\geq 3$, $b_{j-1}\leq\frac{1}{3}$ and $b_j<1$, then:
\begin{equation*}
B \leq\left( \begin{matrix}{\frac{1}{2}} &{\frac{1}{3}} \\
1 & 0
\end{matrix} \right)
\end{equation*}
Calculating the spectral radius
of $W=\left( \begin{matrix}{\frac{1}{2}} &{\frac{1}{3}} \\
1 & 0
\end{matrix} \right)$, we find that it is $\rho(W) <1$, then there exists $0<\lambda_1<1$, and $C>0$ such that  $$\|W^n\|_{\infty}<C\lambda_1^n$$
In particular, $\lim_{n\to\infty}\|W^n\|_{\infty}=0 $.\\
In conclusion, if we consider a product of length $s$ of type $BBB\cdots BBB$, then there exists $0<\lambda_1<1$ and $C_1>0$ such that

\begin{equation}\label{eq:prodB}
\|B^s\|_{\infty}\leq C_1\lambda_1^s.
\end{equation}

\item We study now the product of matrices of type $U_{i_1}\cdots U_{i_{s-1}}U_2$ with $i_1,\dots,i_{s-1}\in\{1,2\}$.\\
We observe that if $a_j\geq M$, then there exists $\epsilon>0$ such that

\begin{equation}\label{eq:epsilon}
U_2=A_{\ell}(j)\leq\left( \begin{matrix}{\frac{1}{\ell}+\frac{1-b_j}{\ell-1}} &{\epsilon} \\
1 & 0
\end{matrix} \right)\leq\left( \begin{matrix}{\frac{1}{3}+\frac{1-b_j}{2}} &{\epsilon} \\
1 & 0
\end{matrix}\right).
\end{equation}
We continue to work in the limit case supposing that $\epsilon=0$ and studying the different possibilities of length two (recall always that $\ell\geq 3$  and $0<b_{j}\leq\frac{1}{3}$ ).
\begin{itemize}
\item If $a_j\geq M$ and $a_{j+2}\geq M$, then
\begin{eqnarray*}
U_1U_2&\leq&\left( \begin{matrix}{\frac{1}{\ell}+\frac{1-b_{j+1}}{\ell-1}} &{b_{j}} \\
1 & 0
\end{matrix} \right)U_2
\leq\left( \begin{matrix}{\frac{5}{6}} &{b_{j}} \\
1 & 0
\end{matrix} \right)\left( \begin{matrix}{\frac{1}{3}+\frac{1-b_{j}}{2}} &{0} \\
1 & 0
\end{matrix} \right)\\
&=&\left( \begin{matrix}{\frac{25}{36}+\frac{7}{12}b_j} &{0} \\[6pt]
{\frac{5}{6}-\frac{b_j}{2}} & 0
\end{matrix} \right)
\leq \left( \begin{matrix}{\frac{32}{36}} &{0} \\[6pt]
{\frac{5}{6}} & 0
\end{matrix} \right)\\
&\leq&\frac{8}{9}\left( \begin{matrix}{1} &{0} \\
{1} & 0
\end{matrix} \right).
\end{eqnarray*}
In particular
\begin{eqnarray}\label{eq:2U1}
\|U_1U_2\|_{\infty}&\leq&\frac{8}{9}.
\end{eqnarray}
\item If $a_{j}\geq M$ and $a_{j+1}\geq M$ then:
\begin{eqnarray*}
U_2U_2&\leq& \left( \begin{matrix}{\frac{5}{6}} &{0} \\
{1} & 0
\end{matrix} \right)\left( \begin{matrix}{\frac{5}{6}} &{0} \\
{1} & 0
\end{matrix} \right)=\left( \begin{matrix}{\frac{25}{36}} &{0} \\[6pt]
{\frac{5}{6}} & 0
\end{matrix} \right)
=\frac{5}{6} U_2,
\end{eqnarray*}
\begin{eqnarray}\label{eq:2U2}
\|U_2U_2\|_{\infty}&\leq&\frac{5}{6}
\end{eqnarray}
\end{itemize}

We consider now the product of $s$ matrices of the form $U_{i_1}\cdots U_{i_{s-1}}U_2$ with $i_1,\dots,i_{s-1}\in\{1,2\}$. We observe that we can find at most $\frac{s}{2}$ for $s$ even ($\frac{s-1}{2}$ for $s$ odd) couples of the type $U_1U_2$ and $U_2U_2$. What remains is some isolated $U_2$ which has norm equal to $1$ (we observe that we can't have some isolated $U_1$ because of $U_1U_1=U_2U_1$) .\\
So, by (\ref{eq:2U1}) and (\ref{eq:2U2}) there exists $0<\lambda_2<1$ such that, if $s$ is even

\begin{eqnarray}\label{eq:prodU}
\|U_{i_1}\cdots U_{i_{s-1}}U_2\|_{\infty}&\leq& \left(\lambda_2\right)^{\frac{s}{2}}
\end{eqnarray}
and if $s$ is odd
\begin{eqnarray}\label{eq:prodUU}
\|U_{i_1}\cdots U_{i_{s-1}}U_2\|_{\infty}&\leq& \left(\lambda_2\right)^{\frac{s-1}{2}}
\end{eqnarray}

\item We consider now the case of a product of $n_1$ matrices \\$A_{\ell}(j_{n_1})A_{\ell}(j_{n_{1}-1})\cdots A_{\ell}(j_{1})$ such that $a_{j_{1}}, a_{j_{n_1}+1}\geq M $ and $a_{j_2},a_{j_3},\dots,a_{j_{n_1}}<M$ (we are considering products of the type $U_1BB\cdots BBU_2$ ) Under these hypotheses, in the limit case we have that
\begin{eqnarray*}
U_2=A_{\ell}(j_{1})&\leq& \left( \begin{matrix}{\frac{5}{6}} &{0} \\
{1} & {0 }
\end{matrix} \right),
\end{eqnarray*}
\begin{eqnarray*}
U_1=A_{\ell}(j_{n_1})&\leq& \left( \begin{matrix}{\frac{5}{6}} &{\frac{1}{3}} \\
{1} & 0
\end{matrix} \right)
\end{eqnarray*}
and for all $k\in\lbrace{2,n_1-1}\rbrace$
\begin{eqnarray*}
B=A_{\ell}(j_{k})&\leq& \left( \begin{matrix}{\frac{1}{2}} &{\frac{1}{3}}. \\
{1} & 0
\end{matrix} \right)
\end{eqnarray*}

We assume that $n_1$ is even and $n_1>2$. We have that:

\begin{eqnarray*}
U_1B\dots\dots BU_2&\leq&  \left( \begin{matrix}{\frac{5}{6}} &{\frac{1}{3}} \\
{1} & 0
\end{matrix} \right)\left( \begin{matrix}{\frac{1}{2}} &{\frac{1}{3}} \\
{1} & 0
\end{matrix} \right)^{n_1-4}\left( \begin{matrix}{\frac{1}{2}} &{\frac{1}{3}} \\
{1} & 0
\end{matrix} \right)^{2}\left( \begin{matrix}{\frac{5}{6}} &{0} \\
{1} & 0
\end{matrix} \right)\\
&\leq&  \frac{5}{6}\left( \begin{matrix}{\frac{5}{6}} &{\frac{1}{3}} \\
{1} & 0
\end{matrix} \right)\left( \begin{matrix}{\frac{1}{2}} &{\frac{1}{3}} \\
{1} & 0
\end{matrix} \right)^{n_1-4}\left( \begin{matrix}{\frac{1}{2}} &{\frac{1}{3}} \\
{1} & 0
\end{matrix} \right)\left( \begin{matrix}{1} &{0} \\
{1} & 0
\end{matrix} \right)\\
&\leq& \frac{5}{6}\left( \begin{matrix}{\frac{5}{6}} &{\frac{1}{3}} \\
{1} & 0
\end{matrix} \right)\left( \begin{matrix}{\frac{1}{2}} &{\frac{1}{3}} \\
{1} & 0
\end{matrix} \right)^{n_1-4}\left( \begin{matrix}{\frac{5}{6}} &{0} \\
{1} & 0
\end{matrix} \right)\\
&\leq& \left(\frac{5}{6}\right)^\frac{n_1-2}{2}\left( \begin{matrix}{\frac{5}{6}} &{\frac{1}{3}} \\
{1} & 0
\end{matrix} \right) \left( \begin{matrix}{\frac{5}{6}} &{0} \\
{1} & 0
\end{matrix} \right)\\
&\leq&  \left(\frac{5}{6}\right)^\frac{n_{1}-2}{2}\left( \begin{matrix}{\frac{37}{36}} &{0} \\
{1} & 0
\end{matrix} \right)\\
&\leq& \frac{37}{36} \frac{5}{6}\left(\frac{5}{6}\right)^\frac{n_{1}-4}{2}\left( \begin{matrix}{1} &{0} \\
{1} & 0
\end{matrix} \right)\\
&\leq&\left(\frac{8}{9}\right)^\frac{n_{1}-2}{2}\left( \begin{matrix}{1} &{0} \\
{1} & 0
\end{matrix} \right).
\end{eqnarray*}

We assume now that $n_1$ is odd, $n_1\geq 3$, then:

\begin{eqnarray*}
U_1B\dots\dots BU_2&\leq&  \left( \begin{matrix}{\frac{5}{6}} &{\frac{1}{3}} \\
{1} & 0
\end{matrix} \right)\left( \begin{matrix}{\frac{1}{2}} &{\frac{1}{3}} \\
{1} & 0
\end{matrix} \right)\left( \begin{matrix}{\frac{1}{2}} &{\frac{1}{3}} \\
{1} & 0
\end{matrix} \right)^{n_1-3}\left( \begin{matrix}{\frac{5}{6}} &{0} \\
{1} & 0
\end{matrix} \right)\\
&\leq&  \left(\frac{5}{6}\right)^\frac{n_1-3}{2}\left( \begin{matrix}{\frac{5}{6}} &{\frac{1}{3}} \\
{1} & 0
\end{matrix} \right) \left( \begin{matrix}{\frac{1}{2}} &{\frac{1}{3}} \\
{1} & 0
\end{matrix} \right)\left( \begin{matrix}{\frac{5}{6}} &{0} \\
{1} & 0
\end{matrix} \right)\\
&\leq&  \left(\frac{5}{6}\right)^\frac{n_{1}-1}{2}\left( \begin{matrix}{\frac{5}{6}} &{\frac{1}{3}} \\
{1} & 0
\end{matrix} \right)\left( \begin{matrix}{1} &{0} \\
{1} & 0
\end{matrix} \right)\\
&\leq&  \frac{7}{6}\frac{5}{6}\left(\frac{5}{6}\right)^\frac{n_{1}-3}{2}\left( \begin{matrix}{1} &{0} \\
{1} & 0
\end{matrix} \right)\\
&\leq&  \left(\frac{35}{36}\right)^\frac{n_{1}-1}{2}\left( \begin{matrix}{1} &{0} \\
{1} & 0
\end{matrix} \right).
\end{eqnarray*}

So we can conclude that for each sequence of $n_1$ matrices $U_1B\dots BU_2$, there exists $0<\lambda_3<1$ such that

 \begin{eqnarray}\label{eq:prodUBUinp}
\|U_1BB\cdots BBU_2\|_{\infty}\leq\lambda_3^{\frac{n_1-2}{2}}.
\end{eqnarray}
 \end{itemize}
After these first observations, we fix $n\in\N$, and $2\leq i\leq n-1$ and we estimate the following quantity:
 $$\left\|\prod_{j = i}^n A_{\ell}(j)\right\|_{\infty}.$$

 \begin{lemma}\label{lem:prodas}
 Let $n\in\N$, $2\leq i\leq n-1$ and $A:=\prod_{j = i}^n A_{\ell}(j)$.\\
 Then there exists a constant $s>0$ such that for each term $A_{\ell}(k)\circ\dots\circ A_{\ell}(k-s)$ of $A$ of length $s$,
 $$\|A_{\ell}(k)\dots A_{\ell}(k-s)\|_{\infty}\leq \omega$$
with $\omega\in(0,1)$.
 \end{lemma}
 \begin{proof}
We consider a term $A_{\ell}(k_r)\circ\dots\circ A_{\ell}(k_1)$ of $A$.
 \begin{itemize}
 \item If for all $k_1<j<k_r$ $A_{\ell}(j)=B$, then by (\ref{eq:prodB}) there exist two constants $C>0$ and $\lambda_1\in(0,1)$ such that:
  $$\|A_{\ell}(k_r)\dots A_{\ell}(k_1)\|_{\infty}\leq C\lambda_1^r. $$
We observe that we have the same estimation for a sequence of the type $U_1\underbrace{B\dots B}_{r}U_2$.
 \item If $A_{\ell}(k_r)\circ\cdots\circ A_{\ell}(k_1)$ is composed first of a sequence of the type $B\cdots BU_2$, then of sequences of the type $U_1B\cdots BU_2$ or $U_{\{1,2\}}\cdots U_2$ and lastly of a sequence of the type $U_1B\cdots B$, then by (\ref{eq:prodB}), (\ref{eq:prodU}), (\ref{eq:prodUU}) and (\ref{eq:prodUBUinp}) there exists $C>0$ and $\lambda_2\in(0,1)$ such that $ \|\underbrace{BB\cdots BB}_{n_1}U_2\underbrace{\cdots}_{n_2}\cdots\underbrace{\cdots}_{n_k}U_1\underbrace{BB\cdots BB}_{n_{k+1}}\|_{\infty}$ is less or equal than
 \begin{eqnarray*}
C \lambda_2^{n_1}\lambda_2^{n_2}\cdots\lambda_2^{n_k}C\lambda_2^{n_{k+1}}= C^2\lambda_2^{n_1+n_2+\dots+n_k+n_{k+1}}=C^2\lambda_2^{\tilde{r}}.
 \end{eqnarray*}
We denote now $\lambda=\max\lbrace{\lambda_1, \lambda_2\rbrace}$ and $s>0$ such that $\max\lbrace{C\lambda^s, C^2\lambda^s\rbrace}<\omega<1$.
  \end{itemize}
 \end{proof}

Finally, by Lemma \ref{lem:prodas} there exist $s>0$ and $\omega\in(0,1)$ such that, if $ \prod_{j = i}^n A_{\ell}(j)$ contains $k$ terms of length $s$, then:
\begin{equation}\label{eq:finestsperons}
\left\|\prod_{j = i}^n A_{\ell}(j)\right\|_{\infty}\leq C\omega^{k}
\end{equation}
with the constant $C>0$ which is the upper bound for the last term of length $n-i-ks +1$.
\par

We observe that the estimations found are valid in the limit case, we have in fact used the hypothesis that $\epsilon=0$ (see (\ref{eq:epsilon})). In the general case, we can use the continuity of the norm $\|\cdot\|_{\infty}$ as a function of $\epsilon$ in order to have the same type of estimation as in (\ref{eq:finestsperons}).\\

In conclusion, by inequality (\ref{vit}), the sequence $\nu_n=-\log \beta_n$ is bounded and by Remark \ref{finpr} and Propositions \ref{n7} and \ref{n1}, the proof of the second claim of Theorem \ref{sw1} is complete.

\section{Proof of Theorem \ref{3}}
\subsection{Basic Definitions and Facts about Cherry flows}
In this section we provide some definitions and properties concerning Cherry flows. We will state them in a compact form. More comments and details can be found in Chapter $4$ of \cite{my4}.

\begin{definition}
A Cherry flow is a $\Cinf$ flow on the torus $\T$ without closed trajectories which has exactly two singularities, a sink and a saddle, both hyperbolic.
\end{definition}
The first example of such a flow was given by Cherry in $1938$, see \cite{Cherry}.
\begin{proposition}
Let $X$ be a Cherry flow and let $\Sing(X)$ be the set of the singularities of $X$. There exists a closed $\Cinf$ curve, $C$ on $\T\setminus \Sing(X)$ with the following properties:
\begin{itemize}
\item $C$ is everywhere transversal to the flow;
\item $C$ is not retractable to a point.
\end{itemize}
\end{proposition}
\begin{fact}
Every Cherry flow admits a Poincar\`e section that can be identified with the unit circle $\S$.
\end{fact}
\begin{fact}\label{critexp}
Let $X$ be a Cherry flow with Poincar\`e section $\S$. The first return function $f$ to $\S$ belongs to our class $\L$ with flat interval $U$ which is composed by all points of $\S$ which are attracted by the singularities. Moreover if $\lambda_1>0>\lambda_2$ are the two eigenvalues of the saddle point, then $f$ has critical exponent $\ell=\frac{|\lambda_2|}{\lambda_1}$.
\end{fact}

\begin{definition}
Let $\gamma$ be a non-trivial recurrent trajectory. Then the closure $\overline{\gamma}$ of $\gamma$ is called a quasi-minimal set.
\end{definition}
\begin{fact}\label{prod}
Every Cherry field has only one quasi-minimal set, which is locally homeomorphic to the Cartesian product of a Cantor set $\Omega$ and a segment $I$. Moreover $\Omega$ is equal to the non-wandering set\footnote{the set of the points $x$ such that for any open neighborhood $V\ni x$ there exists an integer $n>0$ such that the intersection of $V$ and $f^n(V)$ is non-empty.} of the first return function $f$ that is $\Omega=\S\setminus\cup_{i=0}^{\infty}f^{-i}(U)$.
\end{fact}
By Fact \ref{prod}, Fact \ref{critexp} and the product formula for the Hausdorff dimension, in order to prove Theorem \ref{3} we need to calculate the Hausdorff dimension of $\Omega$ \footnote{For more details the reader can refer to the proof of Theorem 1.6 in \cite{my}. The procedure is exactly the same.}. More precisely we need the following result:
\begin{theorem}\label{mare}
Let $f$ be a function in $\L$ with critical exponent $\ell\geq 3$, then the non-wandering set $\Omega$ has Hausdorff dimension strictly greater than zero.
\end{theorem}

\subsection{Proof of Theorem \ref{mare}}

Theorem \ref{mare} is a generalization of Theorem 1.5 in \cite{my} for functions in $\L$ with any rotation number. The proof is now much more simplified.
\paragraph{Standing assumptions.} In this section we always deal with functions in $\L$ with critical exponent $\ell\geq 3$. Moreover we will often use the structure of the dynamical partitions $\mathscr{P}_{n}$ of the circle explained in Subsection \ref{dynamicalpartition}. Because of the symmetry of the functions in $\L$ we will always assume that $n \in \N$ is even. The case $n \in \N$ being odd is completely analogous.

\begin{proposition}\label{aiuto2}
Every long gap of $\mathscr P_{n-1}$ is comparable with the first and the last gap of $\mathscr P_{n}$ which appear in its subdivision.
\end{proposition}

\begin{proof}
Let $I_{0}^n$ and $I_{1}^n$ be the first and the last gap of the subdivision of a long gap $I^{n-1}$ of the partition $\mathscr P_{n-1}$.
Without loss of generality we may assume that
\begin{eqnarray*}
I^{n-1}=(\underline{0}, \underline{-q_{n-1}}),
\end{eqnarray*}
\begin{eqnarray*}
I_{0}^n=(\underline{0}, \underline{-q_{n+1}}),
\end{eqnarray*}
and
\begin{eqnarray*}
I_{1}^n=(\underline{-q_{n-1}-q_n}, \underline{-q_{n-1}}).
\end{eqnarray*}

For the first gap $I_{0}^n$, the proof is easy and comes directly from the second claim of Theorem \ref{sw1}, in fact
\begin{eqnarray*}
 \frac{|(\underline{0},\underline{-q_{n+1}})|}{|(\underline{0}, \underline{-q_{n-1}})|}=\tau_{n+1}.
\end{eqnarray*}

For $I_{1}^n$, we apply Fact \ref{F1} and we find a constant $C_1>0$ such that
\begin{eqnarray*}
\frac{|I_{1}^n|}{|I^{n-1}|}&\geq& C_1 \frac{|(\underline{-q_{n-1}-q_n+1}, \underline{-q_{n-1}+1})|}{|(\underline{1}, \underline{-q_{n-1}+1})|}
\end{eqnarray*}

Apply now Proposition \ref{Koebe liv} to
\begin{itemize}
\item[-]$T=[\underline{-q_{n-2}-q_{n-1}+1}, \underline{-q_{n-1}+1}]$,
\item[-]$J=(\underline{-q_{n-2}-q_{n-1}+1}, \underline{-q_{n-1}+1})$,
\item[-]$f^{q_{n-1}-1}$.
\end{itemize}
and we get a positive constant $C_2$ such that
\begin{eqnarray*}
\frac{|I_{1}^n|}{|I^{n-1}|}&\geq&  C_1 C_2 \frac{|(\underline{-q_n}, \underline{0})|}{|(\underline{q_{n-1}}, \underline{0})|}.
\end{eqnarray*}

Finally by Proposition \ref{figo1} and by the second claim of Theorem \ref{sw1} we find two constants $C_3>0$ and $C_4>0$ such that
\begin{eqnarray*}
\frac{|I_{1}^n|}{|I^{n-1}|}\geq C_1 C_2 C_3\frac{|(\underline{-q_n}, \underline{0})|}{|[\underline{-q_{n}}, \underline{0})|}\geq C_1 C_2 C_3 C_4.
\end{eqnarray*}

\end{proof}

\paragraph{ Proof of Theorem \ref{mare}.}

\begin{proof}
By Subsection \ref{dynamicalpartition} we know that the dynamical partition $\mathscr P_n$ contains two types of gaps: short and long. We define $\{\mathscr G_{n}\}_{n\geq 1}$ a sequence of subsets of $\mathscr P_{n}$ (i.e. $\mathscr G_{n}\subset \mathscr P_{n}$) using induction. Initialise by putting in $\mathscr P_{1}$ a long gap of $\mathscr G_{1}$. Any short gap of $\mathscr P_{n-1}$ contained in $\mathscr G_{n-1}$ is put into $\mathscr G_{n}$ (now it is a long gap of $\mathscr P_{n}$). For a long gap $I^{n-1}_i \in \mathscr P_{n-1}$ belonging to $\mathscr G_{n-1}$ we recall that it is split into a set of gaps of $\mathscr P_{n}$. The left-most, $I^{n+1}_i$, and right-most, $I^n_{i+q_{n-1}+q_{n}}$, enter $\mathscr G_{n}$. The former is a short gap of $\mathscr P_{n}$ and the latter is a long one.

Following this we construct inductively a sequence of probability measures $\{\mu_n\}_{n\geq1}$ on $(\S,\mathcal A_n)$, where $\mathcal A_n$ is the algebra generated by $\mathcal G_n$. We put $\mu_1(I)=1$ for $I\in \mathscr G_{1}$. For the gaps appearing in division of $I^{n-1}_i$ (which is the long gap of $\mathscr P_{n-1}$) we set
\begin{equation}\label{fi}
\mu_n(I^{n+1}_i) = \mu_n(I_{i+q_{n-1}+q_{n}}^{n})=\frac{\mu_{n-1}(I_{i}^{n-1})}{2}.
\end{equation}
The measure of the short gaps remains unchanged. By the construction $\mu_n$ restricted to $\mathcal A_{n-1}$ coincides with $\mu_{n-1}$, thus, by Carath\'eodory's theorem, there exists a probability measure $\mu$ on $(\S,\sigma(\mathscr A_{1},\mathscr A_{2},..))$ extending $\{\mu_n\}_{n\geq1}$. One easily checks that $\mu$ is supported by $K=\cap_{n\in\N}\cup_{I^{n}_i\in\mathscr G_{n}}I^{n}_i \subset \Omega$. Moreover, we have $\mu\left(I_{i}^{n}\right)\leq\frac{1}{2}^\frac{n}{2}$ for any $I^{n}_i \in \mathscr G_{n}$ and also by Proposition \ref{aiuto2} and Corollary \ref{cor:trouzero}, $\lambda_{1}^{n}\leq\left|I_{i}^{n}\right|\leq\lambda_{2}^{n}$ for two constants $\lambda_1,\lambda_2 \in (0,1)$. Therefore,
\begin{equation}\label{ne1}
\mu\left(I_{i}^{n}\right)\leq\left|I_{i}^{n}\right|^{\alpha}
\end{equation}
with $\alpha=\log_{\lambda_1}\frac{1}{\sqrt{2}}>0$.

Let $I$ be an arbitrary interval and let $I_{i}^{n}$ be an element of $\mathscr G_{n}$ contained in $I$ with $n$ as small as possible. Then $I$ is covered by at most two elements of $\mathscr G_{{n-1}}$, $I_{j}^{n-1}$ and $I_{j'}^{n-1}$ and by a set of $\mu$-measure zero (since it is contained in $\S\setminus K$). By (\ref{fi}), Proposition \ref{aiuto2} and  (\ref{ne1}) the following is true:
\begin{equation}\label{mis}
\mu\left(I\right)\leq\mu\left(I_{j}^{n-1}\right)+\mu\left(I_{j'}^{n-1}\right)\leq C_1\mu\left(I_{j}^{n-1}\right)\leq C_2\mu\left(I_{i}^{n}\right)\leq C_2 \left|I_{i}^{n}\right|^{\alpha}\leq C_2\left|I\right|^{\alpha}.
\end{equation}
Finally, let $\mathscr{K}$ be an $\epsilon$-cover of the set $K$. By inequality (\ref{mis}),
\begin{displaymath}
\sum_{I\in\mathscr K}\left|I\right|^{\alpha}\geq\frac{1}{C_2}\sum_{I\in\mathscr K}\mu\left(I\right)\geq \frac{1}{C_2}\mu\left(K\right)=\frac{1}{C_2}.
\end{displaymath}
Finally, the Hausdorff dimension of $K$, and thus also the one of $\Omega$, is strictly bigger than zero.

\end{proof}

\subsection*{Acknowledgments}
I am very grateful to Prof. J. Graczyk, who introduced me to the topic of the paper and subsequently supported me during  the work with his knowledge and kind encouragement. I would like also to thank Prof. M. Rams for many helpful discussions and suggestions which led to the current version of the proof of Theorem \ref{mare}.
%
%

\section{Appendix}
In the following we prove the first claim of Theorem \ref{sw1}. The proof is a generalization of Theorem A.2 and Lemma A.12 in \cite{my} to the case of functions in $\L$ with any rotation number.
\\

In this section we will always work with the following sequence:
\begin{displaymath}
\alpha_n=\frac{|(\underline{-q_n},\underline{0})|}{|[\underline{-q_n},\underline{0})|}.
\end{displaymath}
Since $\forall n \in \N$, $\alpha_n>\tau_n$, we shall prove the first claim of Theorem \ref{sw1} for the sequence $(\alpha_n)_{n\in \N}$.

\par
Let
\begin{displaymath}
\sigma_n=\frac{|(\underline{0},\underline{q_n})|}{|(\underline{0},\underline{q_{n-1}})|}
\end{displaymath}
and
\begin{displaymath}
s_n=\frac{|[\underline{-q_{n-2}},\underline{0}]|}{|\underline{0}|}.
\end{displaymath}
We have the following Theorem:

\begin{theorem}
There exists a natural number $N \in \N$, such that for $n > N$ we have the following inequality:
\begin{equation}\label{alfa}
(\alpha_n)^\ell\leq\prod_{i=0}^{a_n-1}K_{i,n}C_n\tilde{M_n}(l)\alpha_{n-2}^2
\end{equation}
where for all $i\in\{0,\dots,a_n-1\}$, if we denote:
\[
	\tau_{i,n}=\max_{j \in \{0,\dots,q_{n-1}-2\}}|f^{j}((\underline{iq_{n-1}+1},\underline{-q_{n-1}+1}])|
\]
and
\[
	\rho_n=\max_{j \in \{0,\dots,q_{n-2}-2\}}|f^{j}((\underline{a_nq_{n-1}+1},\underline{1}])|,
\]
then
\begin{equation}\label{K_n}
K_{i,n}=e^{\sigma(\tau_{i,n})\sum_{j=0}^{q_{n-1}-2}|f^{j}(\underline{-q_n+iq_{n-1}+1})|},
\end{equation}
\[
	C_n=e^{\sigma(\rho_n)\sum_{j=0}^{q_{n-2}-2}|f^{j}(\underline{-q_{n-2}+1})|}
\]
and
\[
\tilde{M_n}(l)=s_{n-1}^2\cdot\frac{2}{l}\cdot\left(\frac{1}{1+\sqrt{1-\frac{2(l-1)}{l}C_ns_{n-1}\alpha_{n-1}}}\right)\cdot\frac{1}{1-\alpha_{n-2}}\cdot\frac{\sigma_n}{\sigma_{n-2}}.
\]
\end{theorem}

\begin{proof}
The proof is exactly the same of Theorem A.2 in \cite{my} (pag. 150). In fact, the author doesn't use any assumption on the rotation number.
\end{proof}

We prove now the convergence of the sequence $(\alpha_n)_{n \in \N}$.
\par
In \cite{my} (see Lemma A.12 and continuation, pag. 153-154), without any assumption on the rotation number, the author proves that $\prod_{m=0}^{n}C_m$ converges and  that $\prod_{m=0}^{n}\tilde{M_m}$ tends to zero. It remains to study the convergence of the product
$$\prod_{m=0}^{n}\prod_{i=0}^{a_m-1}K_{i,m}$$

which is assured by the following Lemma:

\begin{lemma}\label{esp}
There exists $0<\lambda<1$ such that, for all $i\in\{0,\dots,a_n-1\}$ and for $m$ big enough,
$$\prod_{i=0}^{a_m-1}K_{i,m}\leq e^{\sigma(\lambda^{m-3})\lambda^{m-2}}.$$
\end{lemma}

\begin{proof}
By the order of the preimages of the flat interval on the circle (see Subsection \ref{dynamicalpartition}) and the definition of $K_{i,m}$ (see \ref{K_n} ) we can observe that:

\begin{enumerate}
\item each interval $f^{j}((\underline{iq_{m-1}+1},\underline{-q_{m-1}+1}])$ is contained in a gap of the partition $\mathscr{P}_{m-3}$,
\item for all $i$, $\sum_{j=0}^{q_{m-1}-2} | f^{j}(\underline{-q_m+iq_{m-1}+1}) |$ is contained in a gap of $\mathscr P_{m-1}$ and each of the two sums  $\sum_{j=0}^{q_{m-1}-2} | f^{j}(\underline{-q_m+iq_{m-1}+1}) |$\\ and  $\sum_{j=0}^{q_{m-1}-2} | f^{j}(\underline{-q_m+i'q_{m-1}+1}) |$ is disjoint.  Moreover the total sum\\
$\sum_{i=0}^{a_m-1}\sum_{j=0}^{q_{m-1}-2}|f^{j}(\underline{-q_m+iq_{m-1}+1})|$ is contained in a gap of the partition $\mathscr{P}_{m-2}$.

\end{enumerate}
Since the lengths of the gaps of the partition $\mathscr P_m$ tend to zero at least exponentially fast by Corollary \ref{cor:trouzero}, then there exists $0<\lambda<1$ such that, for $m$ big enough,

\begin{displaymath}
	\sigma\left(\max_{j \in \{0,\dots,q_{m-1}-2 \}}|f^{j}((\underline{iq_{m-1}+1},\underline{-q_{m-1}+1}])|\right)<\sigma(\lambda^{m-3})
\end{displaymath}

and

\begin{displaymath}
\sum_{i=0}^{a_m-1}\sum_{j=0}^{q_{m-1}-2}|f^{j}(\underline{-q_m+iq_{m-1}+1})|<\lambda^{m-2}.
\end{displaymath}
Finally $\prod_{i=0}^{a_m-1}K_{i,m}$ is equal to
\begin{multline*}
	 \exp\left(\sum_{i=0}^{a_m-1}\sigma\left(\max_{j\in\{0,\dots,q_{m-1}-2\}}|f^{j}((\underline{iq_{m-1}+1},\underline{-q_{m-1}+1}])|\right)\right.\\\left.\sum_{j=0}^{q_{m-1}-2}|f^{j}(\underline{-q_m+iq_{m-1}+1})|\right)	 
\end{multline*}
%
which is strictly less than
\begin{eqnarray*}
 \exp \left( \sigma(\lambda^{m-3})\lambda^{m-2} \right).
\end{eqnarray*}
\end{proof}

In conclusion, for $\ell\leq 2$ the inequality \eqref{alfa} implies that the sequence $(\alpha_n)_{n \in \N}$ tends to zero at least exponentially fast. Since $\alpha_n>\tau_n$, we have the same result for the sequence $(\tau_n)_{n\in \N}$.

\end{document}